\documentclass[12pt,oneside,american,english]{amsart}
\usepackage{lmodern}

\usepackage{courier}
\usepackage[T1]{fontenc}
\usepackage[latin9]{inputenc}
\usepackage[letterpaper]{geometry}
\usepackage{prettyref}
\usepackage{float}
\usepackage{amsthm}
\usepackage{amstext}
\usepackage{graphicx}

\makeatletter

\floatstyle{ruled}
\newfloat{algorithm}{tbp}{loa}
\providecommand{\algorithmname}{Algorithm}
\floatname{algorithm}{\protect\algorithmname}

\numberwithin{equation}{section}
\numberwithin{figure}{section}
\theoremstyle{plain}
\newtheorem{thm}{\protect\theoremname}
  \theoremstyle{plain}
  \newtheorem{lem}[thm]{\protect\lemmaname}
  \theoremstyle{remark}
  \newtheorem*{rem*}{\protect\remarkname}
  \theoremstyle{plain}
  \newtheorem{prop}[thm]{\protect\propositionname}

\usepackage{algpseudocode}

\@ifundefined{showcaptionsetup}{}{%
 \PassOptionsToPackage{caption=false}{subfig}}
\usepackage{subfig}
\makeatother

\usepackage{babel}
  \addto\captionsamerican{\renewcommand{\lemmaname}{Lemma}}
  \addto\captionsamerican{\renewcommand{\propositionname}{Proposition}}
  \addto\captionsamerican{\renewcommand{\remarkname}{Remark}}
  \addto\captionsamerican{\renewcommand{\theoremname}{Theorem}}
  \addto\captionsenglish{\renewcommand{\lemmaname}{Lemma}}
  \addto\captionsenglish{\renewcommand{\propositionname}{Proposition}}
  \addto\captionsenglish{\renewcommand{\remarkname}{Remark}}
  \addto\captionsenglish{\renewcommand{\theoremname}{Theorem}}
  \providecommand{\lemmaname}{Lemma}
  \providecommand{\propositionname}{Proposition}
  \providecommand{\remarkname}{Remark}
\addto\captionsamerican{\renewcommand{\algorithmname}{Algorithm}}
\providecommand{\theoremname}{Theorem}

\begin{document}

\title{Barycentric Hermite Interpolation}

\author{Burhan Sadiq and Divakar Viswanath}

\thanks{NSF grants DMS-0715510, DMS-1115277, and SCREMS-1026317.}

\email{bsadiq@umich.edu and divakar@umich.edu}
\begin{abstract}
Let $z_{1},\ldots,z_{K}$ be distinct grid points. If $f_{k,0}$ is
the prescribed value of a function at the grid point $z_{k}$, and
$f_{k,r}$ the prescribed value of the $r$\foreignlanguage{american}{-th}
derivative, for $1\leq r\leq n_{k}-1$, the Hermite interpolant is
the unique polynomial of degree $N-1$ ($N=n_{1}+\cdots+n_{K}$) which
interpolates the prescribed function values and function derivatives.
We obtain another derivation of a method for Hermite interpolation
recently proposed by Butcher et al. {[}\emph{Numerical Algorithms,
vol. 56 (2011), p. 319-347}{]}. One advantage of our derivation is
that it leads to an efficient method for updating the barycentric
weights. If an additional derivative is prescribed at one of the interpolation
points, we show how to update the barycentric coefficients using only
$\mathcal{O}\left(N\right)$ operations. Even in the context of confluent
Newton series, a comparably efficient and general method to update
the coefficients appears not to be known. If the method is properly
implemented, it computes the barycentric weights with fewer operations
than other methods and has very good numerical stability even when
derivatives of high order are involved. We give a partial explanation
of its numerical stability.
\end{abstract}
\maketitle

\section{Introduction}

An example of a Hermite interpolant in barycentric form is the unique
polynomial $\pi(z)$ of degree $3$ such that $\pi(-1)=f_{-1}$, $\pi'(-1)=f_{-1}^{'}$,
$\pi(1)=f_{1}$, and $\pi'(1)=f_{1}^{'}$ as given by

\begin{eqnarray}
\pi(z) & = & (z-1)^{2}(z+1)^{2}\Biggl(f_{-1}\left(\frac{1}{4(z+1)^{2}}+\frac{1}{4(z+1)}\right)+\frac{f_{-1}^{'}}{4(z+1)}\nonumber \\
 &  & +f_{1}\left(\frac{1}{4(z-1)^{2}}-\frac{1}{4(z-1)}\right)+\frac{f_{1}^{'}}{4(z-1)}\Biggr).\label{eq:hermite-example-pm1}
\end{eqnarray}
 In general Hermite interpolation, the function value and its first
$n_{k}-1$ derivatives are prescribed as 
\[
f_{k,0},\ldots,f_{k,n_{k}-1}
\]
at the interpolation point or grid point $z_{k}$ for each $z_{k}$
from the list $z_{1},\ldots,z_{K}$. The problem is to find a polynomial
$\pi(z)$ of degree $N-1$, where $N=n_{1}+\cdots+n_{K}$, such that
\begin{equation}
\frac{d^{r}\pi(z)}{dz^{r}}\Biggl|_{z=z_{k}}=f_{k,r}\quad\text{for}\quad r=0,\ldots,n_{k}-1\label{eq:hermite-conditions}
\end{equation}
at each grid point $z_{k}$. We shall always assume the grid points
$z_{k}$ to be distinct. For an elegant proof of the existence and
uniqueness of the Hermite interpolant, see \cite{Davis1975}.

The polynomial $\pi(z)$ can be represented in either the Newton form
or the barycentric form. In the Newton form, the grid points $z_{k}$
must be ordered in some way \cite{ContedeBoorBook}. If the grid points
are not carefully ordered, the Newton form is susceptible to catastrophic
numerical instability \cite{FischerReichel1989,TalEzer1991}. In contrast,
the barycentric form does not require the grid points to be ordered
and treats all the grid points equally.

\textbf{Barycentric form. }The barycentric form of the interpolant
is 

\begin{equation}
\begin{split}\pi(z)=\pi^{\ast}(z)\sum_{k=1}^{K}\frac{f_{k,n_{k}-1}}{(n_{k}-1)!}\left(\frac{w_{k,0}}{(z-z_{k})}\right)+\frac{f_{k,n_{k}-2}}{(n_{k}-2)!}\left(\frac{w_{k,0}}{(z-z_{k})^{2}}+\frac{w_{k,1}}{(z-z_{k})}\right)+\cdots\\
+f_{k,0}\left(\frac{w_{k,0}}{(z-z_{k})^{n_{k}}}+\cdots+\frac{w_{k,n_{k}-1}}{(z-z_{k})}\right)
\end{split}
\label{eq:barycentric-form-1}
\end{equation}
where $\pi^{\ast}(z)$ is defined as $\prod_{k=1}^{K}(z-z_{k})^{n_{k}}$.
It is evident from inspection that $\pi(z)$, as represented in \prettyref{eq:barycentric-form-1},
is a polynomial of degree $N-1$. For a unique choice of weights $w_{k,r}$,
$\pi(z)$ will satisfy the interpolation conditions \prettyref{eq:hermite-conditions}.
Determining and updating the weights (or coefficients) $w_{k,r}$
of the barycentric form is the topic of this paper.

The representation of $\pi(z)$ shown in \prettyref{eq:barycentric-form-1}
is usually termed the first barycentric form. For the second barycentric
form, take $f_{k,0}=1$ and $f_{k,r}=0$ for $r\geq1$ corresponding
to the function $f(z)=1$. By the existence and uniqueness of the
Hermite interpolant, we have $1=\pi^{\ast}(z)\sum_{k=1}^{K}\sum_{r=0}^{n_{k}-1}w_{k,r}(z-z_{k})^{n_{k}-r}$.
Dividing \prettyref{eq:barycentric-form-1} by the barycentric representation
of $1$, we have 
\begin{equation}
\pi(z)=\frac{\sum_{k=1}^{K}\sum_{s=0}^{n_{k}-1}\frac{f_{k,s}}{s!}\sum_{r=0}^{n_{k}-s-1}w_{k,r}(z-z_{k})^{n_{k}-r-s}}{\sum_{k=1}^{K}\sum_{r=0}^{n_{k}-1}w_{k,r}(z-z_{k})^{n_{k}-r}}.\label{eq:barycentric-form-2}
\end{equation}
This is the second barycentric form. The second barycentric form has
a useful property that we now state. If the first barycentric form
\prettyref{eq:barycentric-form-1} satisfies the interpolation conditions
\prettyref{eq:hermite-conditions}, so does the second barycentric
form \prettyref{eq:barycentric-form-2}. Less obviously, for any choice
of the weights $w_{k,r}$ with $w_{k,0}\neq0$ for $k=1,\ldots,K$,
the second barycentric form is a rational function of $z$ which satisfies
the interpolation conditions \prettyref{eq:hermite-conditions} (see
\cite{SchneiderWerner1991}and \cite{BerrutBM2005}). The second barycentric
form is more robust in the presence of rounding errors in the weights
$w_{k,r}$ , as one may expect and as we will demonstrate in Section
4. 

For deriving the weights $w_{k,r}$, we shall work with the first
barycentric form \prettyref{eq:barycentric-form-1}. It is obvious
from inspection that either of the two forms can be used to evaluate
the interpolant $\pi(z)$ at a given point $z$ using $\mathcal{O}(N)$
arithmetic operations, although the second barycentric form is more
robust in the presence of rounding errors.

\textbf{Calculating barycentric weights using divided differences.}
One of the methods for finding the weights $w_{k,r}$ in $\mathcal{O}(N^{2})$
operations is due to Schneider and Werner \cite{SchneiderWerner1991}.
Two ideas go into that method. We briefly describe the two ideas in
the simpler context of Lagrange interpolation (see Werner \cite{Werner1984}
for this special case) where the problem is to find a polynomial $p(z)$
of degree $K-1$ such that $p(z_{k})=f_{k,0}$ holds at each grid
point $z_{k}$. 

The Newton representation of $p(z)$ is $a_{0}+a_{1}(z-z_{1})+a_{2}(z-z_{1})(z-z_{2})+\cdots+a_{K-1}\prod_{k<K}(z-z_{k})$
for some coefficients $a_{k}$. The Lagrange representation takes
the form $\sum_{k=1}^{K}b_{k}L_{k}(z)$, where $L_{k}(z)$ is the
un-normalized Lagrange cardinal function $\prod_{j\neq k}(z-z_{j})$.
The first idea is to note that the coefficients $a_{k}$ and the coefficients
$b_{k}$ are connected by a triangular matrix \cite{MilneThompson2000}.
The second idea is to note that if $p(z)\equiv1$ then $a_{0}=1$
and $a_{k}=0$ for $k\geq1$ and the coefficients $b_{k}$ are equal
to the Lagrange weights $w_{k}$. The triangular relationship between
the two sets of coefficients is inverted using divided differences
to compute the Lagrange weights. The method of Schneider and Werner
for computing the barycentric weights $w_{k,r}$ of the Hermite interpolant
follows the same logic.

We make two comments about the method of Schneider and Werner. Firstly,
because the Newton expansion is used implicitly, the method depends
upon divided differences and finding a good ordering of the grid points
$z_{k}$. Even if the grid points are carefully ordered and scaled
using the logarithmic capacity of the interval of interpolation, it
is unclear from the extant literature if the method is numerically
stable in the presence of high order derivatives. As our discussion
in Sections 4 and 5 will indicate, numerical stability is a more delicate
issue when high order derivatives occur in the Hermite data. 

Secondly, the method for finding barycentric weights described in
Section 2 typically has a lower operation count than the method of
divided differences. However, Schneider and Werner \cite{SchneiderWerner1991}
allow a general denominator in the barycentric form while we specialize
the denominator to be $1$. Even though it could be possible to specialize
the method of divided differences to lower the operation count, that
is unlikely to make it a better method for computing the barycentric
weights. The argument for the method of Section 2 is its simplicity,
directness, and numerical stability.

\textbf{The method of Butcher et al.  \hskip -.1cm \cite{ButcherCorless2011}
and its extension. }In Section 2, we give another derivation of the
method of Butcher et al.  \hskip -.05cm  for computing the barycentric
weights. The method was derived by Butcher et al.  \hskip -.05cm 
using contour integrals and the manipulation of infinite series. Our
derivation is more direct. The formalism of Butcher et al. extends
to Birkhoff interpolation (or Hermite interpolation with incomplete
data), a problem which is not discussed here.

To explain the basic idea we use in Section 2 for computing the barycentric
weights, we go back to the Hermite interpolant shown in \prettyref{eq:hermite-example-pm1}.
We want to compute a cubic polynomial whose Taylor expansion to two
terms with center $z=1$ is equal to $f_{1}+f_{1}^{'}(z-1)$ and with
center $z=-1$ is equal to $f_{-1}+f_{-1}^{'}(z+1)$. In the prefactor
$(z-1)^{2}(z+1)^{2}$, $(z-1)^{2}$ has an effect on the Taylor expansion
at $z=1$ that is easily neutralized by dividing by $(z-1)^{2}$.
The Taylor expansion of $(z+1)^{-2}$ to two terms about $z=1$ is
$1/4-(z-1)/4$. Therefore if $(z+1)^{2}$ is multiplied with $1/4-(z-1)/4$,
the product is $1+\mathcal{O}\left((z-1)^{2}\right).$ Thus the product
\begin{eqnarray*}
(z+1)^{2}(z-1)^{2}\times\frac{1}{(z-1)^{2}}\times\left(\frac{1}{4}-\frac{z-1}{4}\right)\times(f_{1}+f_{1}'(z-1))\\
=(z+1)^{2}(z-1)^{2}\left(f_{1}\left(\frac{1}{4(z-1)^{2}}-\frac{1}{4(z-1)}\right)+\frac{f_{1}^{'}}{4(z-1)}\right)+\mathcal{O}\left((z-1)^{2}\right)
\end{eqnarray*}
is equal to $f_{1}+f_{1}^{'}(z-1)+\mathcal{O}\left((z-1)^{2}\right)$.
This observation along with the fact that the expression above is
evidently $\mathcal{O}\left((z+1)^{2}\right)$ explains part of the
barycentric representation shown in \prettyref{eq:hermite-example-pm1}.

In general, computing barycentric weights comes down to finding the
coefficients in the Taylor polynomial of expressions of the form $\prod_{j\neq k}(z+z_{k}-z_{j})^{-n_{j}}$.
The coefficients of such expressions can be found efficiently and
with good numerical stability using a standard identity from symmetric
function theory.

In Section 3, we show how to update the barycentric weights when an
additional data item is introduced using only $\mathcal{O}\left(N\right)$
operations. Such an efficient update is not available for the Newton
representation of the Hermite interpolant \cite{ContedeBoorBook}.
If the Newton representation orders the grid points as $z_{1},\ldots,z_{K}$,
it can be easily updated if a function value is introduced at a new
grid point $z_{K+1}$ or if an additional derivative is introduced
at $z_{K}$. However, introducing an additional derivative at one
of the points $z_{1},\ldots,z_{K-1}$ will require a lot of the divided
differences table to be recomputed. Indeed, if the additional derivative
is introduced at $z_{1}$, the entire divided differences table has
to be recomputed. The method in Section 3 handles all these cases
using only $\mathcal{O}(N)$ operations.

\textbf{Numerical stability. }The second barycentric interpolant \prettyref{eq:barycentric-form-2}
of the Runge function $1/(1+z^{2})$ in the interval $-1\leq z\leq1$
using $K=512$ grid points and $n-1=47$ derivatives has a numerical
error of about $10^{-15}$ in the $\infty$-norm and in double precision
arithmetic. If the first barycentric form is used, the error is less
than $10^{-12}$ everywhere except very close to the endpoints. The
first barycentric form is very sensitive near the endpoints $z=\pm1$.
At the endpoints, the error in the first barycentric form is as large
as $10^{-7}$, a phenomenon that is discussed in Section 4. Here we
mention that the endpoints $\pm1$ are not included among the $512$
grid points.

It is probably unwise to restrict the discussion of numerical stability
to the interpolation error $||\pi(z)-f(z)||_{\infty}$. This is particularly
true for the second barycentric form which is quite good at producing
an accurate interpolant even if the barycentric weights $w_{k,r}$
themselves are inaccurate. In Section 4, we use extended precision
computations to check the relative errors in the barycentric weights.
We find that the barycentric weights are computed with good accuracy.

The key to the accuracy of the barycentric weights appears to be the
behavior of a triangular system that appears in the method of Section
2 (see Lemma \ref{lem:The-quantities-Ir} in particular). In Section
5, we prove that the barycentric weights are computed with small relative
error in a few limited situations.

\textbf{Barycentric Lagrange interpolation. }Berrut and Trefethen
\cite{BerrutTrefethen2004} have given a new exposition of the barycentric
form of Lagrange interpolation, emphasizing the usefulness of separating
the computation of the Lagrange weights from the evaluation of the
interpolant. They have pointed out that the emphasis on Newton interpolation
found in almost every textbook on numerical analysis is misplaced.
The barycentric form is as efficient as the Newton form and has better
numerical stability. For an application to the computation of finite
difference weights, see \cite{SadiqViswanath2011}. 

Both the first barycentric form \prettyref{eq:barycentric-form-1}
and second barycentric form \prettyref{eq:barycentric-form-2} of
the Hermite interpolant separate the computation of the weights $w_{k,r}$
from the evaluation of the interpolant. There are significant differences
from the Lagrange setting, however. Firstly, the computation of the
barycentric weights is a great deal more complicated than in the Lagrange
setting. The efficient updating of barycentric weights when new items
of data are introduced is even more involved. Finally, the second
barycentric form is more accurate than the first barycentric form
in the Hermite setting as we show in Section 4 (see Figure \ref{fig:Errors-fst-snd}).
In contrast, in the Lagrange setting, the first barycentric form is
backward stable whereas the second barycentric from has a more limited
type of forward stability \cite{Higham2004}.

\section{Barycentric weights}

The Hermite interpolant $\pi(z)$ is the unique polynomial of degree
$N-1$ which satisfies the $N-1$ interpolation conditions \prettyref{eq:hermite-conditions}.
The number of interpolation conditions at the grid point $z_{k}$,
$1\leq k\leq K$, is $n_{k}$ and $N=n_{1}+\cdots+n_{K}$. The polynomial
$\pi^{\ast}(z)$ defined as $\prod_{k=1}^{K}(z-z_{k})^{n_{k}}$ is
of degree $N$. The Hermite interpolation conditions can be reworded
as requiring the Taylor expansion of $\pi(z)$ centered at $z=z_{k}$
to be equal to
\begin{equation}
f_{k,0}+f_{k,1}(z-z_{k})+\cdots+\frac{f_{k,n_{k}-1}}{(n_{k}-1)!}(z-z_{k})^{n_{k}-1}+\mathcal{O}\left((z-z_{k})^{n_{k}}\right)\label{eq:hermite-taylor-conditions}
\end{equation}
at each grid point $z_{k}$. 

Define $\pi_{k}(z)=\pi^{\ast}(z)(z-z_{k})^{-n_{k}}$ so that $\pi_{k}(z)$
is a polynomial of degree $N-n_{k}$.The polynomial 
\begin{equation}
W_{k}(z)=\sum_{r=0}^{r=n_{k}-1}w_{k,r}(z-z_{k})^{r}\label{eq:Wkofz}
\end{equation}
 is defined by the requirement $\pi_{k}(z)W_{k}(z)=1+\mathcal{O}\left((z-z_{k})^{n_{k}}\right).$
In other words, $W_{k}(z)$ is equal to the Taylor series of $\pi_{k}(z)^{-1}$
with center $z=z_{k}$ and truncated at the $n_{k}$-th power.

The polynomial 
\[
\pi_{k}(z)W_{k}(z)\left(f_{k,0}+f_{k,1}(z-z_{k})+\cdots+\frac{f_{k,n_{k}-1}}{(n_{k}-1)!}(z-z_{k})^{n_{k}-1}\right)
\]
 has a Taylor expansion at $z=z_{k}$ which is identical to \prettyref{eq:hermite-taylor-conditions}
because $\pi_{k}(z)W_{k}(z)=1+\mathcal{O}\left((z-z_{k})^{n_{k}}\right)$
in the limit $z\rightarrow z_{k}$. In addition, for $j\neq k$ the
polynomial is $\mathcal{O}\left((z-z_{j})^{n_{j}}\right)$ in the
limit $z\rightarrow z_{j}$ because $\pi_{k}(z)$ is $\mathcal{O}\left((z-z_{j})^{n_{j}}\right)$
in that limit. However, the degree of the polynomial is $N+n_{k}-2$
which is more than $N-1$ if $n_{k}>1$. This difficulty is easily
fixed. We simply need to multiply $W_{k}(z)$, which is thought of
as a polynomial in $(z-z_{k})$, and $\sum_{r=0}^{n_{k}-1}\frac{f_{k,r}}{r!}(z-z_{k})^{r}$
and drop all terms of order $(z-z_{k})^{n_{k}}$ or higher. The following
polynomial
\begin{multline}
\pi_{k}(z)\Biggl(f_{k,0}\left(w_{k,0}+\cdots+w_{k,n_{k}-1}(z-z_{k})^{n_{k}-1}\right)+\cdots\\
+\frac{f_{k,n_{k}-2}(z-z_{k})^{n_{k}-2}}{(n_{k}-2)!}\left(w_{k,0}+w_{k,1}(z-z_{k})\right)+\frac{f_{k,n_{k}-1}(z-z_{k})^{n_{k}-1}}{(n_{k}-1)!}w_{k,0}\Biggr)\label{eq:barycentric-intermediate}
\end{multline}
is of degree $N-1$. For $j\neq k$, it is $\mathcal{O}\left((z-z_{j})^{n_{j}}\right)$
near $z_{j}$ because $\pi_{k}(z)$ is divisible by $(z-z_{j})^{n_{j}}$.
At $z=z_{k}$, it satisfies the interpolation condition \prettyref{eq:hermite-taylor-conditions}.
We have the following lemma.
\begin{lem}
The barycentric weights $w_{k,r}$, $r=0,1,\ldots,n_{k}-1$, are the
coefficients of the unique polynomial $W_{k}(z)$ of degree $n_{k}-1$,
shown in \prettyref{eq:Wkofz}, which satisfies $\pi_{k}(z)W_{k}(z)=1+\mathcal{O}\left((z-z_{k})^{n_{k}}\right)$
in the limit $z\rightarrow z_{k}$.\label{lem:The-barycentric-weights}
\end{lem}
To calculate the barycentric weights, it is useful to look at the
following equation
\begin{equation}
\frac{1}{\left(1-\frac{z}{\alpha_{1}}\right)\left(1-\frac{z}{\alpha_{2}}\right)\ldots\left(1-\frac{z}{\alpha_{N}}\right)}=1+\mathcal{I}_{1}z+\mathcal{I}_{2}z^{2}+\cdots\label{eq:I12}
\end{equation}
We assume $\alpha_{k}\neq0$ and $\mathcal{I}_{1},\mathcal{I}_{2},\ldots$
on the right hand side of \prettyref{eq:I12} are defined by expanding
the left hand side in powers of $z$. 

To find an efficient method to compute the $\mathcal{I}_{r}$, we
define $\mathcal{P}_{r}=\sum_{k=1}^{N}\alpha_{k}^{-r}$ and differentiate
\prettyref{eq:I12} with respect to $z$ to get 
\begin{eqnarray*}
\frac{1}{\prod_{k=1}^{N}\left(1-\frac{z}{\alpha_{k}}\right)}\sum_{k=1}^{N}\frac{1}{\alpha_{k}}\left(1-\frac{z}{\alpha_{k}}\right)^{-1} & = & \mathcal{I}_{1}+2\mathcal{I}_{2}z+3\mathcal{I}_{3}z^{2}+\cdots\\
\left(1+\mathcal{I}_{1}z+\mathcal{I}_{2}z^{2}+\cdots\right)\left(\mathcal{P}_{1}+\mathcal{P}_{2}z+\cdots\right) & = & \mathcal{I}_{1}+2\mathcal{I}_{2}z+\cdots
\end{eqnarray*}
Equating coefficients of $z$, we have the following lemma.
\begin{lem}
The quantities $\mathcal{I}_{r}$ defined by \prettyref{eq:I12} are
related to the inverse power sums $\mathcal{P}_{r}=\sum_{k=1}^{N}\alpha_{k}^{-r}$
by 
\begin{eqnarray*}
\mathcal{I}_{1} & = & \mathcal{P}_{1}\\
2\mathcal{I}_{2} & = & \mathcal{P}_{2}+\mathcal{I}_{1}\mathcal{P}_{1}\\
3\mathcal{I}_{3} & = & \mathcal{P}_{3}+\mathcal{I}_{1}\mathcal{P}_{2}+\mathcal{I}_{2}\mathcal{P}_{1}\\
4\mathcal{I}_{4} & = & \mathcal{P}_{4}+\mathcal{I}_{1}\mathcal{P}_{3}+\mathcal{I}_{2}\mathcal{P}_{2}+\mathcal{I}_{3}\mathcal{P}_{1}
\end{eqnarray*}
and so on. \label{lem:The-quantities-Ir}
\end{lem}
Coincidentally, \prettyref{lem:The-quantities-Ir} appears as Lemma
2 in \cite{ButcherCorless2011} and in \cite{Szabados1993}, but with
different proofs. It has been well known in symmetric function theory
for a very long time.
\begin{rem*}
This lemma appears in passing in modern books on combinatorics. Stanley
\cite[p. 396]{Stanley1999} refers to the review article by Vahlen
\cite[1904]{Vahlen1904} for information about the sources of such
identities. The identities appear to be due to L. Crocchi \cite[1880]{Crocchi1880}.
\end{rem*}
If we note that 
\[
\pi_{k}(z)^{-1}=\prod_{j\neq k}(z_{k}-z_{j})^{-n_{j}}\prod_{j\neq k}\left(1-\frac{z-z_{k}}{z_{j}-z_{k}}\right)^{-n_{j}},
\]
Lemmas \ref{lem:The-barycentric-weights} and \ref{lem:The-quantities-Ir}
imply an algorithm for computing $w_{k,r}$. The algorithm is to begin
by computing $C_{k}=\prod_{j\neq k}(z_{k}-z_{j})^{-n_{j}}$ and define
$\mathcal{P}_{r}=\sum_{j\neq k}n_{j}(z_{j}-z_{k})^{-r}$. The quantities
$\mathcal{I}_{r}$ are obtained from Lemma \ref{lem:The-quantities-Ir}
and $w_{k,r}=C_{k}\mathcal{I}_{r}$. The grid points $z_{k}$ are
assumed to be distinct.

\begin{algorithm}
\begin{algorithmic}[1]
\Function{findPI}{$\zeta_1,\ldots,\zeta_K$, $n_1,\ldots,n_K$}
\State{return $\zeta_1^{n_1}\ldots\zeta_K^{n_K}$}
\EndFunction
\Function{nInvertList}{$z_1,\ldots,z_K$, $\zeta_1,\ldots,\zeta_K$}
\State{$\zeta_k=-1/z_k$ for $k=1,\ldots,K$}
\EndFunction
\Function{PowerSums}{$\zeta_1,\ldots,\zeta_K$, $n_1,\ldots,n_K$, 
                    $\mathcal{P}_1,\ldots,\mathcal{P}_{n-1}$}
\State{Temporaries: $t_1,\ldots,t_K$}
\State{$t_k=n_k$ for $1\leq k\leq K$}
\For{$r=1,\ldots,n-1$}
\State{$t_k=\zeta_k t_k$ for $1\leq k\leq K$}
\State{$\mathcal{P}_r=t_1+\cdots+t_K$}
\EndFor
\EndFunction
\Function{InversePoly}{$\mathcal{P}_1,\ldots,\mathcal{P}_{n-1}$,
                       $\mathcal{I}_0,\mathcal{I}_1,\ldots,\mathcal{I}_{n-1}$}
\State{$\mathcal{I}_0=1$}
\For{$r=1,\ldots,n-1$}
\State{$\mathcal{I}_r=\sum_{s=1}^{r}\mathcal{P}_s \mathcal{I}_{r-s}/r$}
\EndFor
\EndFunction
\State{Comment: The grid points $z_k$ must be distinct.}
\Function{HermiteWeights}{$z_1,\ldots,z_K$,$n_1,\ldots,n_K$,
                          $w_{k,r}$ with $1\leq k\leq K$ and $0\leq r < n_k$}
\For{$k=1,\ldots,K$}
\State{Temporaries: $z_1',\ldots,z_{K-1}'$, $n_1',\ldots,n_{K-1}'$,$n$, $C_k$,
                    $\zeta_1,\ldots,\zeta_{K-1}$}
\State{$n=n_k$}
\State{Temporaries: $\mathcal{P}_1,\ldots,\mathcal{P}_{n-1}$,
                    $\mathcal{I}_0,\ldots,\mathcal{I}_{n-1}$}
\State{$z_j'=z_k-z_j$ for $1\leq j < k$ and 
      $z_j'=z_k-z_{j+1}$ for $k\leq j \leq K-1$}
\State{$n_j'=n_j$ for $1\leq j < k$ and
      $n_j'=n_{j+1}$ for $k\leq j \leq K-1$}
\State{\Call{nInvertList}{$z_1',\ldots,z_{K-1}'$,$\zeta_1,\ldots,\zeta_{K-1}$}}
\State{\Call{PowerSums}{$\zeta_1,\ldots,\zeta_{K-1}$,$n_1',\ldots,n_{K-1}'$,
                        $\mathcal{P}_1,\ldots,\mathcal{P}_{n-1}$}}
\State{\Call{InversePoly}{$\mathcal{P}_1,\ldots,\mathcal{P}_{n-1}$,
                          $\mathcal{I}_0,\mathcal{I}_1,\ldots,\mathcal{I}_{n-1}$}}
\State{$C_k$= 1/\Call{findPI}{$z_1',\ldots,z_{K-1}'$,$n_1',\ldots,n_{K-1}'$}}
\State{$w_{k,r}=C_k\mathcal{I}_r$ for $r=0,1,\ldots,n-1$}
\EndFor
\EndFunction
\end{algorithmic}

\caption{Computing barycentric weights\label{alg:Computing-barycentric-weights}}
\end{algorithm}

Pseudo-code derived from a C++ implementation is displayed as Algorithm
\prettyref{alg:Computing-barycentric-weights}. For another derivation
of the same algorithm, see Butcher et al. \cite{ButcherCorless2011}.

The functions \texttt{PowerSums }and \texttt{InversePoly} defined
on lines 7 and 15 of Algorithm \prettyref{alg:Computing-barycentric-weights}
perform $2nK$ and $n^{2}$ arithmetic operations, respectively. Summing
the cost of invoking those functions from lines 30 and 31, we find
that the number of arithmetic operations performed by Algorithm \prettyref{alg:Computing-barycentric-weights}
is $2NK+\sum_{k=1}^{K}n_{k}^{2}$ to leading order. Roughly half of
these operations are additions or subtractions and roughly half are
multiplications. On line 29, a total of $K^{2}$ divisions are performed.
That count can be easily reduced to $K^{2}/2$ because the differences
$z_{k}-z_{j}$ that needed to be inverted are only $K^{2}/2$ in number.
The arithmetic operations on lines 32 and 33 do not affect the leading
order of the total count.

The method of divided differences for computing the barycentric weights
$w_{k,r}$ requires $K(K-1)/2$ divisions and $\frac{N^{2}-\sum n_{k}^{2}}{2}$
multiplications \cite{SchneiderWerner1991}. The number of subtractions
or additions is about the same as the number of multiplications. To
compare the operation counts of the two methods, consider the situation
in which $n_{1}=\cdots=n_{K}=n$. In that situation, the number of
multiplications used by the method of divided differences is $\left(n^{2}K^{2}-Kn^{2}\right)/2$
as against $nK^{2}+n^{2}K$ used by the method of Newton identities.
The ratio of the two quantities is $(K-1)/(2K/n+1)$. The same ratio
applies to the number of subtractions or additions. The two methods
have similar operation counts if $n=2$, but for $n>2$ Algorithm
\prettyref{alg:Computing-barycentric-weights} uses fewer arithmetic
operations by a factor of approximately $n/2$.

The argument for Algorithm \prettyref{alg:Computing-barycentric-weights}
over the method of divided differences is its simplicity and numerical
stability, which will be illustrated in Sections 4 and 5. In addition,
Algorithm \prettyref{alg:Computing-barycentric-weights} has a favorable
operation count.

\section{Updating the barycentric weights}

Suppose the barycentric weights $w_{k,r}$ have been computed assuming
that the function value and $n_{k}-1$ derivatives are prescribed
at grid points $z_{k}$, for $1\leq k\leq K$. Suppose a new item
of data is introduced at the grid point $\zeta$. If $\zeta$ is a
new grid point, the new item of data is the function value at the
new grid point. If $\zeta=z_{k}$ for some $1\leq k\leq K$, the new
item of data must be the $n_{k}$-th derivative at $z_{k}$. In either
case, one new barycentric weight must be computed and all the other
barycentric weights must be updated. We show how to do that in $\mathcal{O}(N)$
operations, where $N=n_{1}+\cdots+n_{K}$.

Suppose new data is introduced at $\zeta$ and $z_{k}\neq\zeta$.
Updating the weights $w_{k,r}$, $0\leq r<n_{k}$, is the easy part.
The polynomial $\pi_{k}(z)$ must be changed to $\pi_{k}(z)(z-\zeta)$.
The weights $w_{k,r}$ were determined such that $\pi_{k}(z)W_{k}(z)=1+\mathcal{O}\left((z-z_{k})^{n_{k}}\right)$
for $z\rightarrow z_{k}$ with $W_{k}(z)=\sum_{r=0}^{n_{k}-1}w_{k,r}(z-z_{k})^{r}$.
If the new weights are $w_{k,r}'$ and the corresponding polynomial
is $W_{k}'(z)$, then we require 
\[
\pi_{k}(z)(z-\zeta)W_{k}'(z)=1+\mathcal{O}\left((z-z_{k})^{n_{k}}\right).
\]
To ensure this condition, it is enough if we find $W_{k}'(z)$ of
degree $n_{k}-1$ such that 
\[
(z-\zeta)W_{k}'(z)=\bigl((z-z_{k})+(z_{k}-\zeta)\bigr)W_{k}'(z)=W_{k}(z)+\mathcal{O}\left((z-z_{k})^{n_{k}}\right)
\]
in the limit $z\rightarrow z_{k}$. This gives the equations
\begin{eqnarray}
(z_{k}-\zeta)w_{k,0}' & = & w_{k,0}\nonumber \\
(z_{k}-\zeta)w_{k,r}'+w_{k,r-1}' & = & w_{k,r}\quad\text{for}\quad1\le r<n_{k}\label{eq:update-wkr}
\end{eqnarray}
which are solved to deduce the updated barycentric weights $w_{k,r}'$.
It takes two operations to update each barycentric weight $w_{k,r}$
for $k$ with $z_{k}\neq\zeta$.

If $z_{k}\neq\zeta$ for $1\leq k\leq K$, then $\zeta$ is the new
grid point $z_{K+1}$. The weights $w_{k,r}$ are updated using \prettyref{eq:update-wkr}
for $1\leq k\leq K$ as already described. We need to compute the
new weight $w_{K+1,0}$. This weight is computed using the formula
$w_{K+1,0}=\prod_{j=1}^{K}(\zeta-z_{k})^{-n_{k}}$. If $\zeta$ is
a new grid point, namely $z_{K+1}$, the total cost for updating all
the weights and computing the new weight is $\mathcal{O}(N)$.

If $\zeta=z_{\kappa}$ for some $\kappa$, $1\leq\kappa\leq K$, the
weights $w_{k,r}$ are updated as already described if $k\neq\kappa$.
We are left with the case of updating the weights $w_{\kappa,r}$
with $\zeta=z_{\kappa}$ for some $\kappa$, $1\leq\kappa\leq K$.
In this case, the new data item is another derivative at $z_{\kappa}$.
Efficient updating of the weights $w_{\kappa,r}$ causes complications
that we now turn to. 

Define 
\begin{eqnarray*}
A_{k} & = & \left\{ 1/(z_{j}-z_{k})\:\text{repeated}\: n_{j}\:\text{times}\:\Bigl|\: j\neq k\:\text{and}\:1\leq j\leq K\right\} \\
C_{k} & = & \prod_{j\neq k}(z_{k}-z_{j})^{-n_{j}}.
\end{eqnarray*}
Here $A_{k}$ is a multiset. Define $\mathcal{P}_{r}(A_{k})$ as the
sum of the $r$-th powers of the elements of $A_{k}$. Let $\mathcal{I}_{r}(A_{k})$
be related to $\mathcal{P}_{r}(A_{k})$ by the triangular identities
of Lemma \ref{lem:The-quantities-Ir}. According to Algorithm \ref{alg:Computing-barycentric-weights},
the weight $w_{k,r}$ is given by $w_{k,r}=C_{k}\mathcal{I}_{r}(A_{k})$
for $0\leq r<n_{k}$ and $1\leq k\leq K$. When computing the weights
$w_{k,0},\ldots,w_{k,n_{k}-1}$, the $2n_{k}-1$ intermediate quantities
\begin{equation}
C_{k},\mathcal{P}_{1}(A_{k}),\ldots,\mathcal{P}_{n_{k}-1}(A_{k}),\mathcal{I}_{1}(A_{k}),\ldots,\mathcal{I}_{n_{k}-1}(A_{k})\label{eq:intermediate}
\end{equation}
arise. We assume that these intermediate quantities are stored for
each $k$. The total number of intermediate quantities stored is $2N-K$.

The reason for storing these intermediate quantities is as follows.
Suppose a new data entry is introduced at the grid point $\zeta=z_{\kappa}$.
Then a new weight $w_{\kappa,n_{\kappa}}$ needs to be generated.
To calculate that weight, $\mathcal{I}_{n_{\kappa}}(A_{\kappa})$
will be generated using the triangular identities of Lemma \ref{lem:The-quantities-Ir}
(the update rules of \prettyref{eq:update-wkr} are of no use here).
The intermediate quantities are used to calculate $\mathcal{I}_{n_{\kappa}}\left(A_{\kappa}\right)$. 

When new data is introduced at the point $\zeta$, the intermediate
quantities are updated at all grid points $z_{k}$ regardless of whether
$z_{k}\neq\zeta$ or $z_{k}=\zeta$ (so than we can handle a new data
item introduced at $z_{k}$ later on). If $z_{k}\neq\zeta$, some
of the intermediate quantities in the list \prettyref{eq:intermediate}
are updated as follows.
\begin{eqnarray}
C_{k}' & \leftarrow & C_{k}/(z_{k}-\zeta)\nonumber \\
\mathcal{P}_{r}(A_{k}') & \leftarrow & \mathcal{P}_{r}(A_{k})+(\zeta-z_{k})^{-r}\quad\text{for}\quad1\leq r<n_{k}\label{eq:update1}
\end{eqnarray}
where the primes denote updated quantities with $A_{k}'=A_{k}\cup\left\{ 1/(\zeta-z_{k})\right\} $
(this is a multiset union not a set union). To update $\mathcal{I}_{r}$,
we go back to \prettyref{eq:I12}, which defines $\mathcal{I}_{1},\mathcal{I}_{2},\ldots$,
and deduce the following:
\begin{equation}
\mathcal{I}_{r}(A_{k}')-\frac{\mathcal{I}_{r-1}(A_{k}')}{\zeta-z_{k}}=\mathcal{I}_{r}(A_{k})\quad\text{for}\quad r=1,\ldots,n_{k}-1\label{eq:update2}
\end{equation}
where it is assumed that $\mathcal{I}_{0}=1$. Using \prettyref{eq:update1}
and \prettyref{eq:update2}, it costs two operations to update each
intermediate quantity in the list \prettyref{eq:intermediate}. Since
the total number of intermediate quantities is $\mathcal{O}(N)$,
the cost for updating the intermediate quantities for $k\neq\kappa$
is also $\mathcal{O}(N)$. 

In $\zeta=z_{\kappa}$, $A_{\kappa}$ does not change and none of
the intermediate quantities in \prettyref{eq:intermediate} with $k=\kappa$
needs to be updated. However, we need to generate $\mathcal{P}_{n_{\kappa}}(A_{\kappa})$
and $\mathcal{I}_{n_{\kappa}}(A_{\kappa})$. These are generated using
the formulas 
\[
\mathcal{P}_{n_{\kappa}}(A_{\kappa})=\sum_{j\neq\kappa}n_{j}\,(z_{j}-z_{\kappa})^{-n_{\kappa}}\quad\mathcal{I}_{n_{\kappa}}(A_{\kappa})=(1/n_{\kappa})\sum_{s=1}^{n_{\kappa}}\mathcal{P}(A_{\kappa})\mathcal{I}_{n_{\kappa}-s}(A_{\kappa})
\]
the first of which is by definition and the second is a triangular
identity of Lemma \ref{lem:The-quantities-Ir}. If the temporaries
$t_{j}$ that arise on line 8 of Algorithm \ref{alg:Computing-barycentric-weights}
are preserved, $\mathcal{P}_{n_{\kappa}}$ can be computed using only
$2n_{\kappa}-1$ operations. The barycentric weights $w_{\kappa,0},\ldots,w_{\kappa,n_{\kappa}-1}$
do not change. The weight $w_{\kappa,n_{\kappa}}$ is computed as
$C_{\kappa}\mathcal{I}_{n_{\kappa}}$. We have thus completed describing
an $\mathcal{O}(N)$ method for updating all the barycentric weights
and generating a new weight when a new data item is introduced.

If the barycentric weights are generated by repeatedly adding new
items of data at grid points, the grid points must be ordered in some
fashion. As we noted in the introductory section, ordering the grid
points may make the method vulnerable to numerical instability. If
the weights need to be updated a small number of times to accommodate
new items of data, the method described in this section will be numerically
stable.

A more general updating problem is to update the existing barycentric
weights and generate new ones, when $n_{k}'$ new data items are introduced
at the grid point $z_{k}$ for $1\leq k\leq K$. Algorithm \ref{alg:Computing-barycentric-weights}
and the updating algorithm of this section can be combined to deduce
an efficient and numerically stable solution to this more general
updating problem.

\section{Illustration of numerical stability}

If the grid points $z_{k}$, $1\leq k\leq K$, and the number of derivatives
at grid points $n_{k}$, $1\leq k\leq K$, are chosen without care,
the barycentric interpolation of $f(z)$ using the Hermite interpolant
$\pi(z)$ can be very badly conditioned. To test the barycentric formulas
\prettyref{eq:barycentric-form-1} and \prettyref{eq:barycentric-form-2},
we first determine a choice of $z_{k},\, n_{k}$ which leads to a
well-conditioned interpolation problem. 

Let $L_{k,r}(z)$ denote the polynomial of degree $N-1$ ($N=n_{1}+\cdots+n_{K}$)
such that 
\[
\frac{d^{s}L_{k,r}(z)}{dz^{s}}\Bigl|_{z=z_{\ell}}=\delta_{r,s}\delta_{k,\ell}\quad\text{for}\quad1\leq\ell\leq K,\:0\leq s<n_{l}.
\]
The fundamental polynomials (or cardinal functions) of Hermite interpolation
$L_{k,r}(z)$ have a prescribed behavior at the grid points. However,
they may become very large in-between the grid points leading to a
poorly conditioned interpolation problem.

If the grid points are chosen as $z_{k}=\cos\left((2k-1)\pi/2K\right)$,
$1\leq k\leq K,$ (these are the Chebyshev points), and $n_{1}=\cdots=n_{K}=n$,
Szabados \cite{Szabados1993} has proved upper bounds on the fundamental
polynomials in the interval $[-1,1]$. In particular,
\[
\sup_{-1\leq z\leq1}\sum_{k=1}^{K}|L_{k,r}(z)|\leq\begin{cases}
\begin{aligned}\begin{aligned}\mathcal{O}\left(\frac{\log K}{K^{r}}\right)\end{aligned}
\quad\text{if}\: n-r\:\text{is odd}\\
\mathcal{O}\left(\frac{1}{K^{r}}\right)\quad\text{if}\: n-r\:\text{is even}
\end{aligned}
\end{cases}
\]
for $r=0,\ldots,n-1$. These bounds imply that a small change in the
interpolation data $f_{k,r}$ will result in a small change in the
interpolant $\pi(z)$ (defined by \prettyref{eq:barycentric-form-1}
or \prettyref{eq:barycentric-form-2}) thus showing the Hermite interpolation
problem to be well-conditioned.

In all our computations, the Chebyshev points $z_{k}$ are replaced
by $2z_{k}$ and correspondingly the data $f_{k,r}$ are replaced
by $f_{k,r}/2^{r}$. If $z\in[-1,1]$, we have 
\[
\lim_{K\rightarrow\infty}\left(\prod_{k=1}^{K}(z-z_{k})\right)^{1/K}=\frac{1}{2}
\]
because the capacity of an interval is a quarter of its length \cite[p. 83]{Davis1975}.
So we may expect the prefactor $\pi^{\ast}(z)$ in the first barycentric
form \prettyref{eq:barycentric-form-1} as well as the quantity $C_{k}$,
which occurs on line 32 of Algorithm \prettyref{alg:Computing-barycentric-weights},
to be roughly of the order $1/2^{nK}$. For reasonably large $K$
and $n$, $C_{k}$ and $\pi^{\ast}(z)$ underflow in IEEE double precision
arithmetic. Replacing $z_{k}$ by $2z_{k}$ and taking $z\in[-2,2]$
changes the limit above to $1$ implying that the quantities such
as $C_{k}$ that occur in the algorithm are better scaled.

Barycentric Hermite interpolation problem is highly susceptible to
overflows and underflows. Scaling the Chebyshev points so that the
capacity of the interval is $1$ as in the previous paragraph is only
a partial cure. Even if the $z_{k}$ are scaled as indicated, a product
evaluated in the order $(2z-2z_{1})$, $(2z-2z_{1})(2z-2z_{2})$,
$(2z-2z_{1})(2z-2z_{2})(2z-2z_{3})$, and so on may overflow or underflow
in its intermediate stages even though the final result can be represented
in machine arithmetic. To decrease the chances of such a thing happening,
we reorder the Chebyshev points $z_{k}$ in the Leja ordering \cite{Reichel1990}
in addition to multiplying them by $2$.

In our implementation of the barycentric formulas \prettyref{eq:barycentric-form-1}
and \prettyref{eq:barycentric-form-2}, the function values input
to the interpolation procedures are $f_{n,r}/r!$ not $f_{n,r}$.
This eliminates the need to first multiply and then divide a quantity
by $r!$ (the computation of which is obviously prone to overflows)
for certain functions such as the Runge function discussed below.

\begin{figure}
\subfloat[]{\centering{}\includegraphics[scale=0.4]{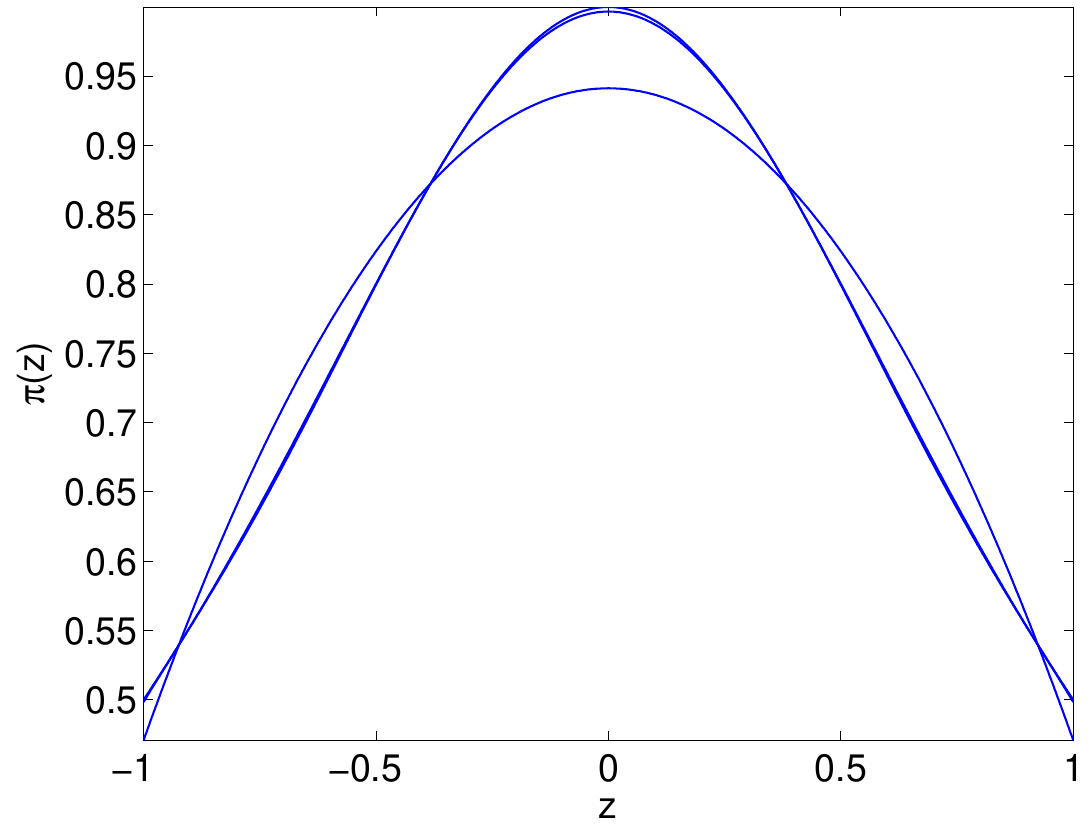}}\subfloat[]{\begin{centering}
\includegraphics[scale=0.4]{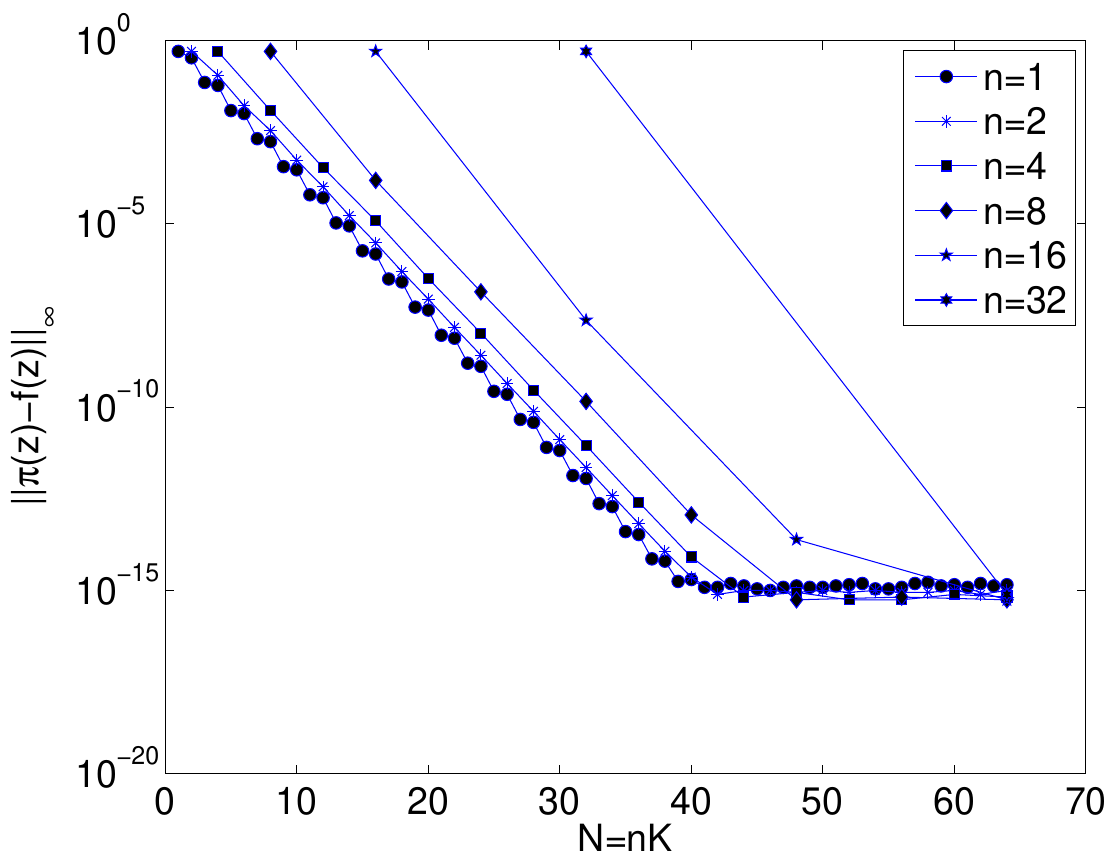}
\par\end{centering}

}

\caption{(A): Interpolants $\pi(z)$ to the Runge function $f(z)=1/(1+z^{2})$
with $K=4$ and $n=1,2,3$. (B): Dependence of interpolation error
on the number of grid points $K$ for fixed number of derivatives
$n$ at each grid point.\label{fig:(A):Runge-interpolation}}
\end{figure}
We now turn to the Runge function $f(z)=1/(1+z^{2})$, which is the
subject of computations shown in Figure \ref{fig:(A):Runge-interpolation}.
A numerically stable method to calculate higher order derivatives
of the Runge function for $z\in[-1,1]$ is implied by the following
calculation:
\begin{eqnarray*}
f(z) & = & \frac{1}{2i}\left(\frac{1}{z-i}-\frac{1}{z+i}\right)\\
\frac{d^{r}f(z)}{dz^{r}} & = & (-1)^{r}\frac{r!}{2i}\left(\frac{1}{(z-i)^{r+1}}-\frac{1}{(z+i)^{r+1}}\right).
\end{eqnarray*}
If $z-i=R\exp(i\theta)$, then we have $f^{(r)}(z)/r!=(-1)^{r-1}R^{r+1}\sin(r+1)\theta$.
This formula has excellent numerical stability for $z\in[-1,1]$.

Figure \ref{fig:(A):Runge-interpolation} shows that the error in
interpolating the Runge function is $c_{n}\exp(-\alpha nK)$. The
positive constant $\alpha$ appears to be independent of $n$. However,
$c_{n}$ increases with $n$. Therefore if we are allowed to use $N$
items of information about the function to form an accurate interpolant,
the best choice for this example is to take them all to be function
values. 

\begin{figure}
\subfloat[]{\centering{}\includegraphics[scale=0.4]{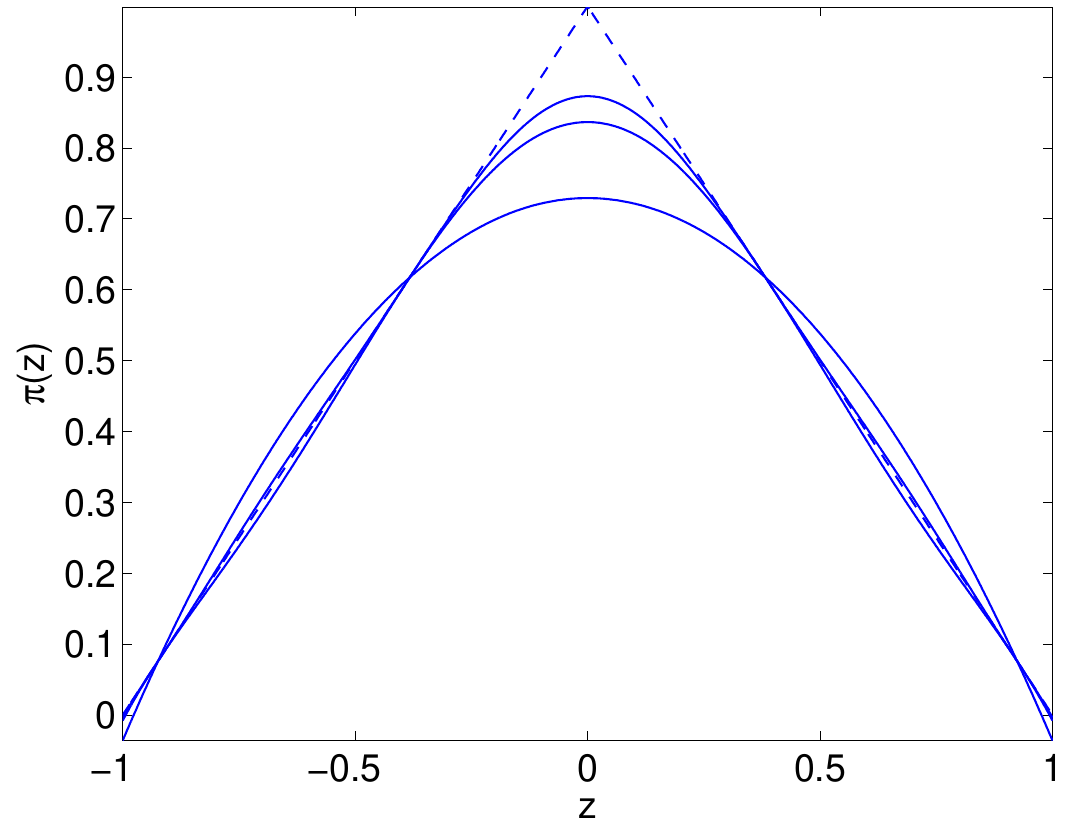}}\subfloat[]{\centering{}\includegraphics[scale=0.4]{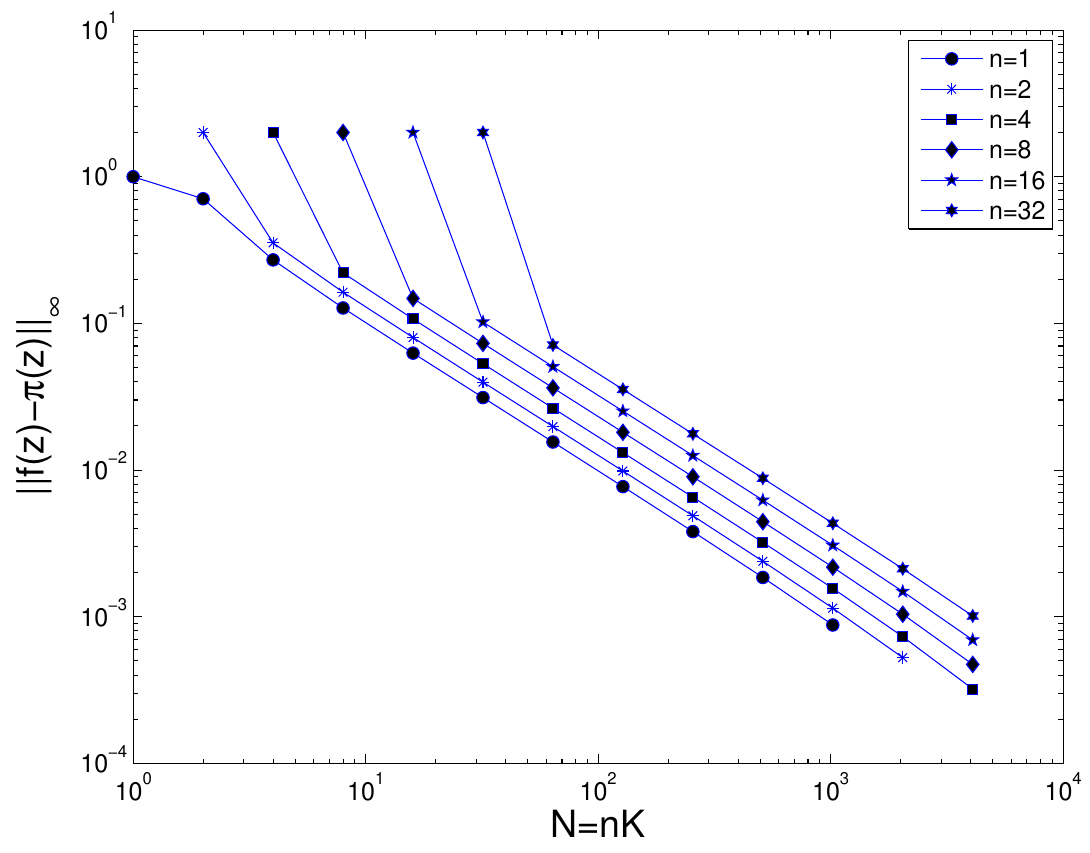}}

\caption{A): Interpolants $\pi(z)$ to the hat function $f(z)=1-|z|$ with
$K=4$ and $n=1,2,3$. (B): Dependence of interpolation error on the
number of grid points $K$ for fixed number of derivatives $n$ at
each grid point.\label{fig:hat-interpolants}}
\end{figure}
For the hat function $f(z)=1-|z|$, Figure \ref{fig:hat-interpolants}
shows that the error is proportional to $1/K$. If the error is represented
in the form $c_{n}/nK$, $c_{n}$ once again increases with $n$.

The Runge function is analytic in an open set around $[-1,1]$, while
the hat function is of bounded variation but not even continuously
differentiable. In the case $n=1$, which corresponds to ordinary
polynomial interpolation, the differing rates of convergence are easily
explained by the difference in smoothness and analyticity \cite{Davis1975}.
That connection likely persists for $n>1$, which corresponds to Hermite
interpolation.

\begin{figure}

\begin{centering}
\includegraphics[scale=0.4]{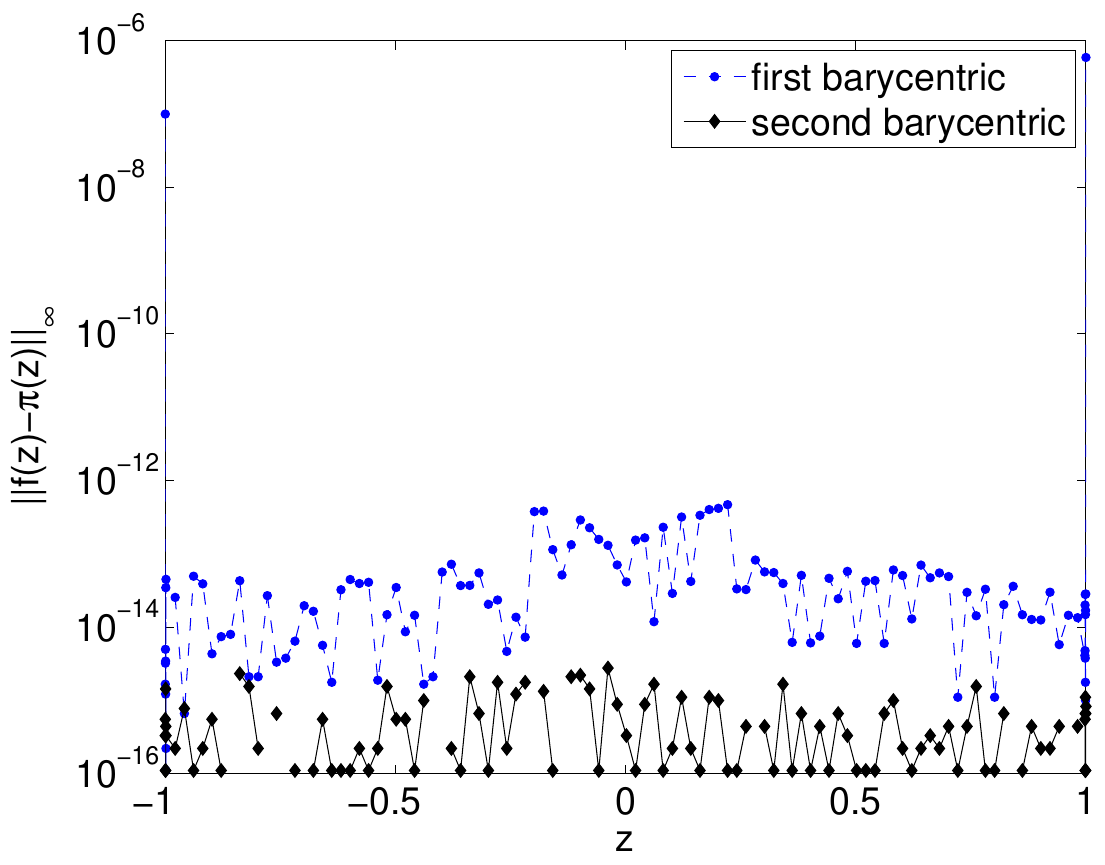}\caption{Interpolation errors to the Runge function $f(z)=1/(1+z^{2})$ with
$K=512$ grid points and $n=48$ derivatives at each grid point. The
first barycentric form is notably inaccurate at $z=\pm1$.\label{fig:Errors-fst-snd}}

\par\end{centering}

\end{figure}
The interpolations discussed so far used the second barycentric form.
Figure \ref{fig:Errors-fst-snd} compares the first barycentric form
\prettyref{eq:barycentric-form-1} to the second \prettyref{eq:barycentric-form-2}.
Even with $K=512$ grid points and $n=48$ derivatives there is no
sign of numerical instability at all. The first barycentric form is
less accurate but has small errors for $z\in(-1,1)$. At $z=\pm1$,
the first barycentric form produces much larger relative errors of
about $10^{-7}$. It is well known that the Chebyshev polynomials
are bounded by $1$ for $z\in[-1,1]$ but increase rapidly outside
the interval. That phenomenon is magnified by several orders for the
fundamental polynomials of Hermite interpolation. Thus we should expect
the Hermite interpolant with $n=48$ to be very sensitive near the
endpoints of the interval. 

Another phenomenon noticeable from Figure \ref{fig:Errors-fst-snd}
is that the errors are somewhat elevated near the center of the interval.
The second barycentric form exhibits neither elevated error near the
middle of the interval nor sensitivity at the end points. It has excellent
numerical stability.

\begin{figure}

\subfloat[]{

\centering{}\includegraphics[scale=0.4]{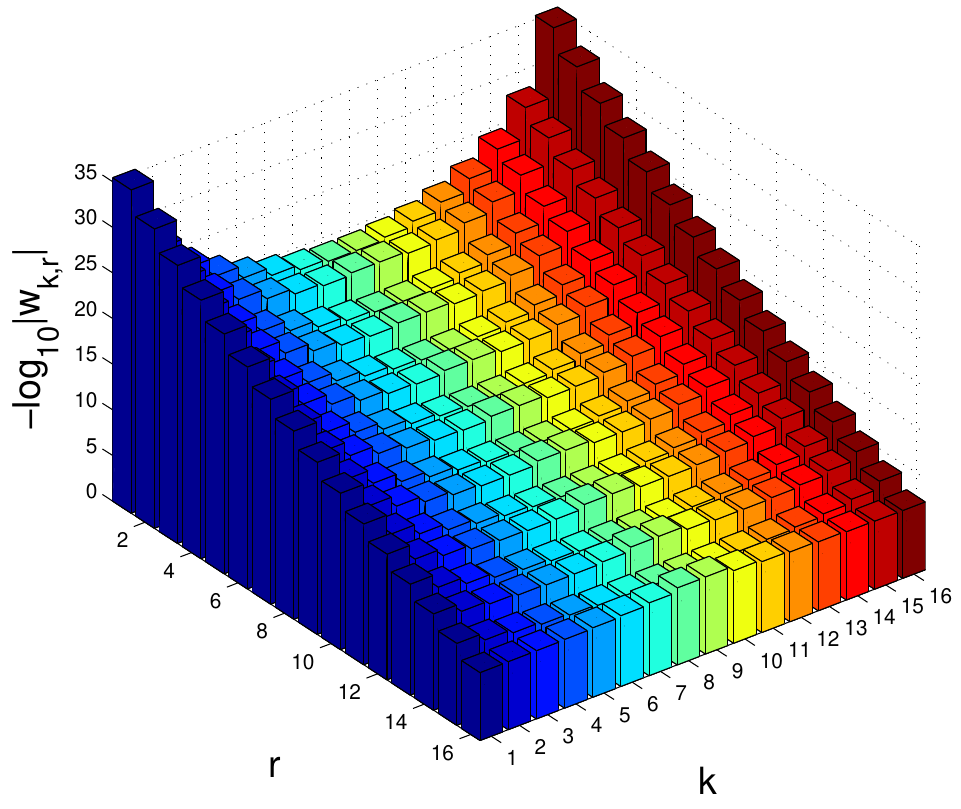}}\subfloat[]{

\begin{centering}
\includegraphics[scale=0.5]{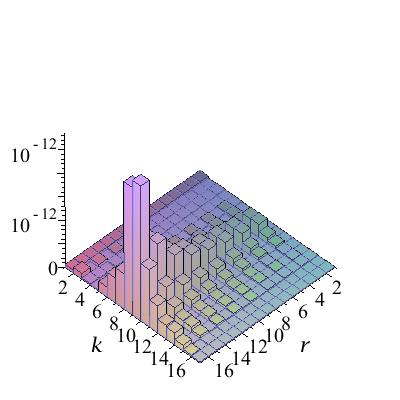}
\par\end{centering}

}\caption{(A) Magnitude of the barycentric weights $w_{k,r}$ and (B) Relative
error in the weights (the maximum relative error is $2.86\times10^{-12}$.
Both plots use $K=16$ grid points and $n=16$ derivatives at each
grid point.\label{fig:Barycentric-weights}}
\end{figure}
In Figure \ref{fig:Barycentric-weights}, we directly examine the
weights used for barycentric interpolation. These weights are computed
using Algorithm \prettyref{alg:Computing-barycentric-weights}. Figure
\ref{fig:Barycentric-weights} uses $K=16$ and $n=16$. The Chebyshev
points $z_{k}$ were multiplied by $2$ but the display in the figure
does not use Leja ordering. From part (A) of the figure, we see that
the weights $w_{k,r}$ are quite small especially when $r=0$. Thus
in spite of using an interval of capacity $1$, the $C_{k}$ that
occur on line 32 of Algorithm \ref{alg:Computing-barycentric-weights}
are getting quite small. Using an interval of capacity $1$ still
allows considerable fluctuations in $\prod_{k=1}^{K}(z-z_{k})$ around
$1$ because it is only the $1/K$-th power which is guaranteed to
converge to $1$. In computing the prefactor $C_{k}$, these fluctuations
are raised to the power $n$, with $n$ being $16$ in this instance. 

From part (B) of Figure \ref{fig:Barycentric-weights}, we see that
the relative errors in the weights $w_{k,r}$ are quite small. However,
the relative errors are the highest for $r=n$ and $z_{k}$ near the
middle of the interval. These phenomena will be partially explained
in the next section.

\begin{figure}

\subfloat[]{\begin{centering}
\includegraphics[scale=0.4]{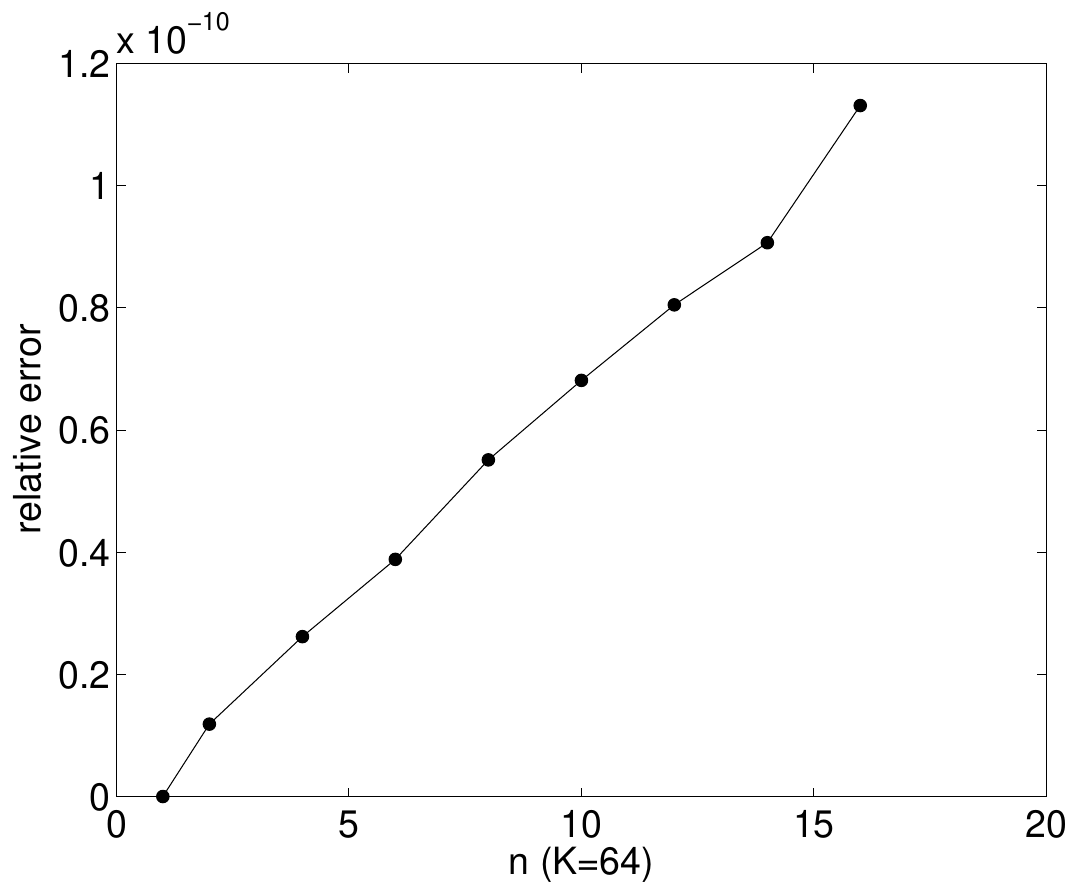}
\par\end{centering}

}\subfloat[]{\begin{centering}
\includegraphics[scale=0.4]{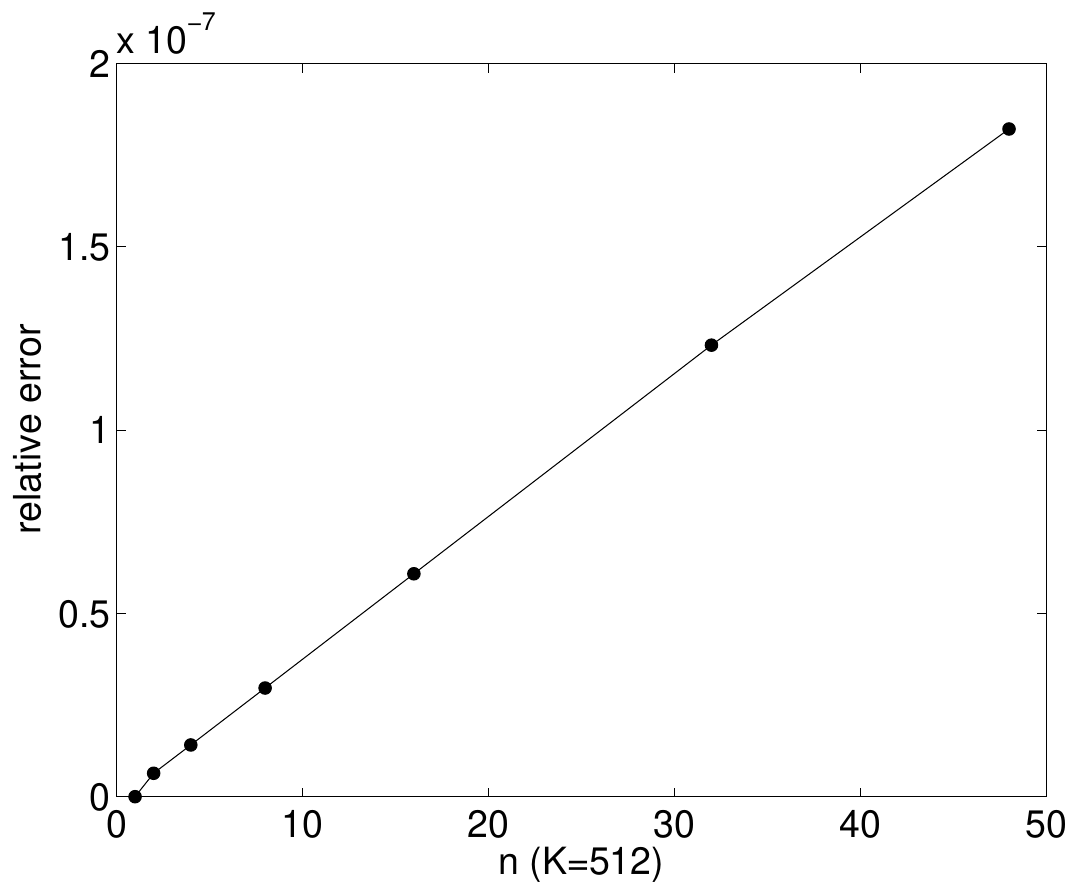}
\par\end{centering}

}\caption{Maximum relative error in the barycentric weights $w_{k,r}$ as a
function of the number of derivatives $n$ at each grid point.\label{fig:Maximum-relative-error-vs-n}}

\end{figure}
In both Figures \ref{fig:Barycentric-weights} and \ref{fig:Maximum-relative-error-vs-n},
the relative errors in the barycentric weights were obtained by comparing
double precision results of a C++ program with extended precision
computations in Maple. Figure \ref{fig:Maximum-relative-error-vs-n}
shows the increase in the maximum relative error in the barycentric
weights $w_{k,r}$ with $n$, which is the number of derivatives used
at each grid point. The increase is mild. Indeed, in part (B) of the
figure, which uses $K=512$ the increase is only linear. Algorithm
\ref{alg:Computing-barycentric-weights} for computing barycentric
weights appears to have very good numerical stability.

\section{Partial explanation of numerical stability}

On line 32 of Algorithm \ref{alg:Computing-barycentric-weights},
the weight $w_{k,r}$ is computed using $w_{k,r}=C_{k}\mathcal{I}_{r}$.
The $C_{k}$ are merely products of differences and will be computed
with excellent relative accuracy. We will restrict ourselves to the
error accumulation in the computation of $\mathcal{I}_{r}$.

We assume that $\left\{ \alpha_{1},\ldots,\alpha_{N}\right\} $ is
a multiset of numbers and that $\mathcal{P}_{r}=\sum_{i=1}^{N}\alpha_{i}^{-r}$.
The quantities $\mathcal{I}_{r}$ are defined by \prettyref{eq:I12}
and computed using $\mathcal{P}_{r}$ and the triangular identities
of Lemma \ref{lem:The-quantities-Ir}. For computing $w_{k,r}$, the
multiset is taken to be $z_{1}-z_{k}$ repeated $n_{1}$ times, $z_{2}-z_{k}$
repeated $n_{2}$ times, and so on---see line 27 of Algorithm \ref{alg:Computing-barycentric-weights}. 

If $u$ is the unit round-off ($u=2^{-53}$ for IEEE double precision
arithmetic), each arithmetic operation of the type $a+b$ will evaluate
to $(a+b)(1+\delta)$ with $|\delta|<u$ in floating point arithmetic.
If $1+\theta_{n}=\prod_{i=1}^{n}(1+\delta_{i})^{\rho_{i}}$, where
$|\delta_{i}|\leq u$ and $\rho_{i}=\pm1$, we have $\theta_{n}\leq\gamma_{n}$,
where $\gamma_{n}$ is defined as $nu/(1-nu)$ (this is Lemma 3.1
of \cite{Higham2002} for whose validity it is assumed that $nu<1$).
This lemma is convenient for accumulating the effect of several rounding
operations. 

Another convenience is the counter notation \cite{Higham2002,Stewart1973}.
In this notation, $\left<n\right>$ stands for a product of the form
$\prod_{i=1}^{n}(1+\delta_{i})^{\rho_{i}}$, where $|\delta_{i}|\leq u$
and $\rho_{i}=\pm1$. Evidently, if $\left<n\right>=1+\delta$ then
$|\delta|\leq\gamma_{n}$ (assuming $nu<1$). In floating point arithmetic,
$(a+b)+c$ evaluates to $a\left<2\right>+b\left<2\right>+c\left<1\right>$.

If $\mathcal{P}_{1}=1/\alpha_{1}+\cdots+1/\alpha_{N}$ is computed
in the order it is written, it evaluates to 
\[
1/\alpha_{1}\left<N\right>+1/\alpha_{2}\left<N\right>+1/\alpha_{3}\left<N-1\right>\cdots+1/\alpha_{N}\left<2\right>.
\]
If $\alpha_{i}$ is perturbed to $\alpha_{i}\left<N-i+2\right>$,
the computed result is the exact result. Thus the computation of $\mathcal{P}_{1}$
is backward stable. However, the computation of $\mathcal{P}_{1}$
can still have large relative errors if there are near cancellations.
In Figure \ref{fig:Maximum-relative-error-vs-n} (A), the relative
error in $\mathcal{P}_{1}$ (the maximum over all $z_{k}$ which occurs
for $z_{k}$ near the middle of the interval) is $1.18\times10^{-11}$.
In part (B), which correspond to $K=512$, the relative error in $\mathcal{P}_{1}$
is $6.42\times10^{-9}$. If there is a point at $0$ and the positive
and negative grid points are symmetric, the exact value of $\mathcal{P}_{r}$
is zero for $r$ an odd number and the relative error is infinite.
However, in this case, $\mathcal{P}_{r}$ will have small relative
error for even $r$ and the only power sums that influence the $\mathcal{I}_{r}$
correspond to even powers. If $r$ is even, $\mathcal{P}_{r}$ is
always computed with excellent relative accuracy as there can be no
cancellations. 

Since the error analysis of $\mathcal{P}_{r}$ is relatively tractable
and in view of the comments made in the previous paragraph, we will
temporarily assume all the power sums $\mathcal{P}_{r}$ to be computed
with small relative errors. The more complex issue is the manner in
which the errors are amplified when the quantities $\mathcal{I}_{r}$
are computed using the triangular identities of Lemma \ref{lem:The-quantities-Ir}.
Figure \ref{fig:Maximum-relative-error-vs-n} shows that the errors
grow only very mildly when going from $\mathcal{P}_{r}$ to $\mathcal{I}_{r}$.

To understand the error growth, we recast the triangular identities
of Lemma \ref{lem:The-quantities-Ir} into matrix form:
\begin{equation}
\left(\begin{array}{ccccc}
1\\
-\mathcal{P}_{1} & 1\\
-\mathcal{P}_{2} & -\mathcal{P}_{1} & 2\\
-\mathcal{P}_{3} & -\mathcal{P}_{2} & -\mathcal{P}_{1} & 3\\
\ddots & \ddots & \ddots & \ddots & \ddots
\end{array}\right)\left(\begin{array}{c}
1\\
\mathcal{I}_{1}\\
\mathcal{I}_{2}\\
\mathcal{I}_{3}\\
\vdots
\end{array}\right)=\left(\begin{array}{c}
1\\
0\\
0\\
0\\
\vdots
\end{array}\right).\label{eq:matrix-system}
\end{equation}
If the condition number of the matrix in \prettyref{eq:matrix-system}
is $\kappa$, the norm-wise relative error in the $\mathcal{I}_{r}$
is bounded by $\frac{2\epsilon\kappa}{1-\epsilon\kappa}$, where $\epsilon$
is of the order of $u$, the unit round-off. Unfortunately, this bound
is useless because $\kappa$ can be very large. If a single $\alpha_{i}$
is less than $1$, which is always the case, the $\mathcal{P}_{r}$
will increase exponentially with $r$. The condition number of the
system \prettyref{eq:matrix-system}, retaining only $n+1$ rows in
that system, will increase exponentially with $n$. 

Although $\kappa$ is quite bad, there is a simple way to get a matrix
system equivalent to \prettyref{eq:matrix-system} with a small condition
number. Suppose the $\alpha_{i}$s are all multiplied by a scaling
factor $s$ and changed to $s\alpha_{i}$. Then $\mathcal{P}_{r}$
and $\mathcal{I}_{r}$ are transformed to $P_{r}=s^{-r}\mathcal{P}_{r}$
and $I_{r}=s^{-r}\mathcal{I}_{r}$, respectively. If $s$ is sufficiently
large then $|P_{r}|\leq1$ for $1\leq r\leq n$.

The $(n+1)\times(n+1)$ matrix obtained by replacing the power sums
$\mathcal{P}_{r}$ by their scaled version $P_{r}$ has condition
number bounded by $n+1$. In general, triangular systems have condition
numbers that increase exponentially even if all the off-diagonal entries
are bounded by $1$ in magnitude \cite{ViswanathTrefethen1998}. Here
the situation is different because the diagonal entries increase in
the order $1,1,2,3,\ldots$ We have the bound
\[
|I_{r}|=\frac{|P_{r}+I_{1}P_{r-1}+\cdots+I_{r-1}P_{1}|}{r}\leq\frac{1+|I_{1}|+\cdots+|I_{r-1}|}{r}\leq1,
\]
 where the last inequality is by induction. Thus all entries in the
inverse will be bounded by $1$.

If a large scaling factor $s$ is used to control the condition number
of the matrix, the difficulty comes in when we use $\mathcal{I}_{n}=s^{n}I_{n}$
to estimate the relative error in $\mathcal{I}_{n}$. The absolute
errors in $I_{r}$ can be bounded by a quantity of the form $n^{5/2}u$,
or even $nu$ if the scaling factor $s$ is chosen to make the power
sums small enough , where $u$ is the unit round-off. However, if
$I_{n}$ is much smaller than $1$ in magnitude, the norm-wise bounds
will imply a large relative error in $I_{n}$ and in $\mathcal{I}_{n}$.
This difficulty can be removed if a scaling factor $s$ can be found
such that $|P_{r}|\leq1$ for $1\leq r\leq n$ and $|I_{n}|$ is $\mathcal{O}(1)$.
The linear growth in relative errors in Figure \ref{fig:Maximum-relative-error-vs-n}
suggests that such a scale might exist. Proving the existence of such
a scale will require estimates of the quantities $\mathcal{P}_{r}$
and $\mathcal{I}_{r}$.

The discussion in this section partially backs up the finding of Figure
\ref{fig:Maximum-relative-error-vs-n} that the growth of relative
errors in the barycentric weights with the order of the derivative
is mild. The contribution to the errors made by $\mathcal{P}_{r}$
is therefore quite significant. Even though the computation of each
power sum $\mathcal{P}_{r}$ is backward stable, the sum itself could
be ill-conditioned. In fact, by looking up $n=1$ in Figure \ref{fig:Maximum-relative-error-vs-n},
we see that the relative error in $\mathcal{P}_{1}$ is not so small.
We implemented compensated summation in x86 assembly language and
verified that it does not reduce  forward errors. 

The following proposition allows us to verify the accuracy of some
of the barycentric weights.
\begin{prop}
If the $\alpha_{i}$ are all positive, the relative error in $\mathcal{P}_{r}$
is bounded by $\gamma_{r+N}$ and the relative errors in $\mathcal{I}_{r}$
are bounded by $\gamma_{3rN}$ for sufficiently small unit round-off
and $r\leq N$. \label{pro:If-the-rel-error}\end{prop}
\begin{proof}
The computed quantity $\mathcal{P}_{r}$ can be represented as 
\[
\alpha_{1}^{-r}\left<N+r\right>+\alpha_{2}^{-r}\left<N+r-1\right>+\cdots+\alpha_{N}^{-r}<r+1>.
\]
Because each $\alpha_{i}$ is positive, the computed quantity can
be represented as $\mathcal{P}_{r}\left<N+r\right>$ thus proving
one half of the proposition.

For the other half, we note that that $\mathcal{I}_{r}=\left(\mathcal{P}_{r}+\mathcal{I}_{1}\mathcal{P}_{r-1}+\cdots+\mathcal{I}_{r-1}\mathcal{P}_{1}\right)/r$.
By induction, the quantity computed for $\mathcal{I}_{r}$ can be
represented as 
\[
\frac{\mathcal{P}_{r}\left<N+r\right>\left<r\right>+\cdots+\mathcal{I}_{r-1}\mathcal{P}_{1}\left<3(r-1)N\right>\left<N+1\right>\left<r\right>}{r}.
\]
In this expression the $\left<r\right>$ factors that accompany each
term in the numerator are caused by the arithmetic operations for
computing $\mathcal{I}_{r}$ using $\mathcal{I}_{s}$ for $1\leq s\leq r-1$.
The factors that accompany each term in the numerator can be multiplied
and replaced by $\left<3Nr\right>$. Since each term is positive,
it follows that the computed quantity can be represented as $\mathcal{I}_{r}\left<3Nr\right>$.
\end{proof}
This proposition implies that the weights $w_{k,r}$ are computed
with excellent accuracy when $z_{k}$ is either the least or the greatest
of the grid points. Indeed, we see from Figure \ref{fig:Barycentric-weights}(B)
that the relative errors near the end points of the interval are very
small. If $z_{k}$ is the least of the grid points, each $z_{j}'$
, $1\leq j\leq K-1$ that occurs on line 27 of Algorithm \ref{alg:Computing-barycentric-weights}
is negative and non-zero. Since $-1/z_{j}'$ corresponds to $\alpha_{j}$
in Proposition \ref{pro:If-the-rel-error}, the proposition is immediately
applicable. If $z_{k}$ is the greatest of the grid points, each $-1/z_{j}'$
is negative and the proposition is not immediately applicable. However,
it is evident from inspection that $\mathcal{P}_{r}$ and $\mathcal{I}_{r}$
are negative or positive according as $r$ is odd or even. Therefore
every sum in the triangular identities of Lemma \ref{lem:The-quantities-Ir}
has terms that are all positive or all negative. Thus the bounds on
relative errors proved in the proposition will apply. 

We have candidly highlighted the inadequacy, in situations not covered
by Proposition \ref{pro:If-the-rel-error}, of the bounds that can
be derived on the rounding errors in the computed barycentric weights.
Thus the title of this section only claims a partial explanation.
Existing results in rounding error analysis known to us give bounds
which are highly pessimistic.

Numerical stability of barycentric Lagrange interpolation has been
studied by Higham \cite{Higham2004}. It may be suspected that the
methods of that paper can be used for the numerical stability analysis
of barycentric Hermite interpolation. However, that is far from being
the case. The definition of the condition number used by Higham perturbs
the function values but not the grid points. Such a definition is
useful in the context of Lagrange interpolation because the Lagrange
weights are obtained from products of the type $\prod_{j\neq k}(z_{j}-z_{k})$.
The basic rules of double precision arithmetic imply right away that
each Lagrange weight is computed with a very small relative forward
error which can be easily bounded. In the case of Hermite interpolation,
the computation of the weights is a great deal more complicated as
we have emphasized. In particular, the computation involves back substitution
or the inversion of the triangular system displayed explicitly in
\prettyref{eq:matrix-system}. 

If good bounds on the forward errors of the barycentric weights $w_{k,r}$
can be obtained, the rounding error analysis simplifies greatly. The
analysis of the first barycentric form \prettyref{eq:barycentric-form-1}
closely parallels the rounding error analysis of inner products of
the type $\sum_{i=1}^{n}a_{i}b_{i}$. For the second barycentric form,
one needs to tackle expressions of the form $A_{1}/A_{2}-A_{3}/A_{4}$,
where each $A_{i}$ has the structure of an inner product. While this
form is slightly more complicated algebraically, it poses none of
the difficulties associated with triangular inversion or back substitution.

Higham's treatise \cite{Higham2002} on numerical algorithms discusses
component-wise analysis of linear systems in Chapter 7 and triangular
inversion in Chapter 8. We now explain why none of the results in
those two chapters help explain the excellent numerical stability
of barycentric Hermite interpolation illustrated in Figure \ref{fig:Maximum-relative-error-vs-n}.
In the component-wise bound of Higham's Theorem 7.4, it is crucial
note the $|A^{-1}|$ that appears in the numerator. Viswanath and
Trefethen have shown \cite{ViswanathTrefethen1998} that the inverses
of triangular matrices generically have norms that are exponential
in the dimension. Therefore such bounds are highly pessimistic for
the situation illustrated in Figure \ref{fig:Maximum-relative-error-vs-n}.
As we described in detail in the arguments leading up to Proposition
\ref{pro:If-the-rel-error}, an attempt at mitigating the norm of
the inverse by scaling the grid points is ultimately of no value.
Exactly the same comments apply to Higham's Lemma 7.9. The difficulty
that one faces is more explicit in Higham's Lemma 8.6, Theorem 8.7,
and Theorem 8.14. In each of those results, the exponential dependence
of the worst case bounds on the dimension of the matrix system is
explicit. If those bounds are applied to \eqref{eq:matrix-system},
the resulting bounds on the errors in the barycentric weights, corresponding
to the two plots in Figure \ref{fig:Errors-fst-snd}, would be off
by factors of $2^{64}$ and $2^{512}$, respectively.

The worst case bounds on rounding errors are exponential in the dimension
of the linear system. The results of Viswanath and Trefethen \cite{ViswanathTrefethen1998}
show that the same is true generically, thus suggesting strongly that
the worst case bounds will be difficult to improve. It is noteworthy
that Higham \cite{Higham2002} begins the chapter on triangular systems
by quoting authorities who realized decades ago that the bounds that
can be derived on the error are usually pessimistic while in practice
many triangular systems can be solved quite accurately. 

We conclude by recommending the second barycentric form \prettyref{eq:barycentric-form-2},
with weights computed using the method of Butcher et al. \hspace{-0.5cm}
\cite{ButcherCorless2011} as presented in Algorithm \ref{alg:Computing-barycentric-weights},
as the standard algorithm for Hermite interpolation. The reasons for
this recommendation are the same as those given by Berrut and Trefethen
\cite{BerrutTrefethen2004} for barycentric Lagrange interpolation,
namely, conceptual simplicity, excellent numerical stability, and
flexibility in incorporating new data points. To those reasons, we
may add the finding of Section 2 that barycentric Hermite interpolation
has a lower operation count than Hermite interpolation using divided
differences.

\section{Acknowledgments}

We thank Prof. Sergey Fomin for references that led us to attribute
Lemma \ref{lem:The-quantities-Ir} to L. Crocchi \cite[1880]{Crocchi1880}.

\bibliographystyle{plain}
\bibliography{references}

\end{document}